%% file: tensor-ver101005.tex
\documentclass[a4paper,12pt]{article}

\usepackage{amsmath,amssymb,amsthm, amscd}

\newcommand{\Fg}{\mathfrak{g}}
\newcommand{\Fh}{\mathfrak{h}}

\newcommand{\CB}{\mathcal{B}}
\newcommand{\CZ}{\mathcal{Z}}
\newcommand{\CK}{\mathcal{K}}
\newcommand{\CO}{\mathcal{O}}

\newcommand{\BC}{\mathbb{C}}
\newcommand{\BR}{\mathbb{R}}

\newcommand{\BZ}{\mathbb{Z}}

\newcommand{\mv}{\mathcal{MV}}

\newcommand{\Hom}{\mathop{\rm Hom}\nolimits}
\newcommand{\wt}{\mathop{\rm wt}\nolimits}
\newcommand{\Conv}{\mathop{\rm Conv}\nolimits}
\newcommand{\Stab}{\mathop{\rm Stab}\nolimits}
\newcommand{\Ad}{\mathop{\rm Ad}\nolimits}

\newcommand{\Gr}{\mathcal{G}r}

\newcommand{\pair}[2]{\langle #1,\,#2 \rangle}
\newcommand{\Bpair}[2]{\bigl\langle #1,\,#2 \bigr\rangle}
\newcommand{\ol}[1]{\overline{#1}}

\newcommand{\ve}{\varepsilon}
\newcommand{\vp}{\varphi}

\newcommand{\bzero}{\mathbf{0}}
\newcommand{\bb}{\mathbf{b}}
\newcommand{\bi}{\mathbf{i}}
\newcommand{\bj}{\mathbf{j}}

\newcommand{\wi}[1]{w^{\bi}_{#1}}
\newcommand{\wj}[1]{w^{\bj}_{#1}}
\newcommand{\si}[1]{s_{i_{#1}}}
\newcommand{\Ni}[1]{n^{\bi}_{#1}}
\newcommand{\Nj}[1]{n^{\bj}_{#1}}

\makeatletter
\@addtoreset{equation}{subsection}

\renewcommand\section{\@startsection{section}{1}{0pt}
{-3.5ex plus -1ex minus -.2ex}{1.0ex plus .2ex}{\large\bf}}
\renewcommand\subsection{\@startsection{subsection}{1}{0pt}
{2.5ex plus 1ex minus .2ex}{-1em}{\bf}}
\makeatother

\newcommand{\vsp}{\vspace{3mm}}

\theoremstyle{plain}
\newtheorem{thm}{Theorem}[subsection]
\newtheorem{lem}[thm]{Lemma}
\newtheorem{prop}[thm]{Proposition}

\newtheorem{fact}[thm]{Fact}

\theoremstyle{definition}
\newtheorem{dfn}[thm]{Definition}

\theoremstyle{remark}
\newtheorem{rem}[thm]{Remark}

\begin{document}

\setlength{\baselineskip}{18pt}

\title{\Large\bf Tensor products and Minkowski sums of \\[3mm]
Mirkovi\'c-Vilonen polytopes}
\author{
 Syu Kato%
 \footnote{Supported in part by JSPS Research Fellowships 
           for Young Scientists (No.\,20740011).} \\
 \small Department of Mathematics, Kyoto University, \\
 \small Kitashirakawa-Oiwake-cho, Sakyo-ku, Kyoto 606-8502, Japan \\
 \small (e-mail: {\tt syuchan@math.kyoto-u.ac.jp}) \\[5mm]
 Satoshi Naito%
 \footnote{Supported in part by Grant-in-Aid for Scientific Research 
           (No.\,20540006), JSPS.} \\ 
 \small Institute of Mathematics, University of Tsukuba, \\
 \small Tsukuba, Ibaraki 305-8571, Japan \ 
 (e-mail: {\tt naito@math.tsukuba.ac.jp})
 \\[2mm] and \\[2mm]
 Daisuke Sagaki%
 \footnote{Supported in part by JSPS Research Fellowships 
           for Young Scientists (No.\,19740004).} \\ 
 \small Institute of Mathematics, University of Tsukuba, \\
 \small Tsukuba, Ibaraki 305-8571, Japan \ 
 (e-mail: {\tt sagaki@math.tsukuba.ac.jp})
}
\date{}
\maketitle

%
\begin{abstract} \setlength{\baselineskip}{16pt}
The purpose of this paper is to prove that 
the Mirkovi\'c-Vilonen (MV for short) polytope corresponding 
to the tensor product of two arbitrary MV polytopes is contained 
in the Minkowski sum of these two MV polytopes. This generalizes 
the result in our previous paper \cite{KNS}, which was obtained 
under the assumption that the first tensor factor is 
an extremal MV polytope.
\end{abstract}
%
%
\section{Introduction.}
\label{sec:intro}
In our previous paper \cite{KNS}, we proved that 
the Mirkovi\'c-Vilonen (MV for short) polytope corresponding 
to the tensor product of two MV polytopes is contained 
in the Minkowski sum of the two MV polytopes, under 
the assumption that the first (i.e., left) tensor factor 
is an extremal MV polytope. The purpose of the present paper 
is to prove the same result for arbitrary two MV polytopes, 
without any assumption on the first tensor factor.

Following the notation of \cite{KNS}, let $G$ be a complex 
semisimple algebraic group with Lie algebra $\Fg$, and $T$ 
a maximal torus with Lie algebra $\Fh$. For a dominant coweight 
$\lambda \in X_{\ast}(T):=\Hom(\BC^{\ast},\,T)$ for $G$, 
let $\mv(\lambda)$ denote the set of all MV polytopes $P$ 
of highest vertex $\lambda$ such that $P \subset 
\Conv(W \cdot \lambda) \subset \Fh_{\BR}:=
\BR \otimes_{\BZ} X_{\ast}(T)$; the set $\mv(\lambda)$ 
gives a realization of the crystal basis $\CB(\lambda)$ 
of the irreducible highest weight module of 
highest weight $\lambda$ over the quantized universal 
enveloping algebra $U_{q}(\Fg^{\vee})$ of 
the (Langlands) dual Lie algebra $\Fg^{\vee}$ of $\Fg$. 

Let $\lambda_{1},\,\lambda_{2} \in X_{\ast}(T)$ be dominant coweights, 
and $P_{1} \in \mv(\lambda_{1})$, $P_{2} \in \mv(\lambda_{2})$. 
If we consider the tensor product
\begin{equation*}
P_{1} \otimes P_{2} \in \mv(\lambda_{1}) \otimes \mv(\lambda_{2})
\cong \CB(\lambda_{1}) \otimes \CB(\lambda_{2}),
\end{equation*}
then there exists a unique dominant coweight $\lambda \in X_{\ast}(T)$ 
and embedding $\iota_{\lambda}:\mv(\lambda) \ (\cong \CB(\lambda)) \hookrightarrow 
\mv(\lambda_{1}) \otimes \mv(\lambda_{2})$ of crystals such that 
$P_{1} \otimes P_{2}=\iota_{\lambda}(P)$ for some $P \in \mv(\lambda)$. 
Now our main result (Theorem~\ref{thm:main}) states that in $\Fh_{\BR}$, 
we have the inclusion $P \subset P_{1} + P_{2}$, 
where $P_{1}+P_{2}$ is the Minkowski sum of the MV polytopes 
$P_{1}$ and $P_{2}$; in \cite{KNS}, we proved the same assertion
under the assumption that $P_{1}$ is an extremal MV polytope. 

We should mention that our method of proof for Theorem~\ref{thm:main} 
is quite different from the one of \cite{KNS}, in which we used 
twisted products of (open dense subsets of) MV cycles in the twisted 
product of two affine Grassmannians. In fact, we make use of a description
(Proposition~\ref{prop:ehrig}) of MV polytopes in terms of Kashiwara 
data for $U_{q}(\Fg^{\vee})$, due to Ehrig, and reduce the problem to 
proving an inequality (Proposition~\ref{prop:key1}) between 
Gelfand-Goresky-MacPherson-Serganova (GGMS for short) data, or equivalently, 
between Berenstein-Zelevinsky (BZ for short) data. This inequality 
(more precisely, Proposition~\ref{prop:apdx} in the Appendix) 
can be regarded as a generalization to an arbitrary semisimple Lie algebra 
of the inequality of \cite[Proposition~2.7]{Kam0} in the case of type $A$;%
\footnote{After finishing this paper, we were informed by Kamnitzer that 
he also had (but, never really wrote down) a proof of this inequality 
in the general case along the lines of our geometric proof in the Appendix}
in \cite{Kam0}, the inequality was used to study the irreducible components 
of certain fibers of the convolution morphism for the affine Grassmannian. 

Since the inequality above is the most important ingredient, 
we provide two different proofs: one proof (given in Subsection~\ref{subsec:keys})
is based on the agreement, due to Kamnitzer \cite{Kam2}, of the 
Lusztig-Berenstein-Zelevinsky (LBZ for short) and 
Braverman-Finkelberg-Gaitsgory (BFG for short) 
crystal structures for $U_{q}(\Fg^{\vee})$ on the set of MV polytopes; 
another proof (given in the Appendix) is a purely geometric one based on 
the original definition of MV polytopes by Anderson \cite{A}. 

Once the inequality above is obtained, Theorem~\ref{thm:main} follows 
easily from the tensor product rule for the action of lowering 
Kashiwara operators for $U_{q}(\Fg^{\vee})$. 

This paper is organized as follows. In Section~\ref{sec:MV}, 
we first recall the basic notation and standard facts concerning 
MV polytopes. Next, we review the relation of MV polytopes 
with MV cycles in the affine Grassmannian, and also the LBZ ($=$ BFG) 
crystal structure on the set of MV polytopes. Furthermore, we give a 
description of MV polytopes in terms of Kashiwara data, due to Ehrig;
we include a short proof of it, which uses Kamnitzer's result.
In Section~\ref{sec:main}, we first state our main result 
(Theorem~\ref{thm:main}). Next, we prove an inequality 
(Proposition~\ref{prop:key1}) between GGMS (or BZ) data, which is a key 
to our proof of Theorem~\ref{thm:main}. Finally, by combining the above 
with the tensor product rule for crystals, we prove Theorem~\ref{thm:main}. 
In the Appendix, using the geometry of the affine Grassmannian, we give 
another proof of the inequality above (or, a slightly strengthened form of it).

%
\section{Mirkovi\'c-Vilonen polytopes.}
\label{sec:MV}

%
\subsection{Basic notation.}
\label{subsec:notation}

Let $G$ be a complex connected semisimple algebraic group, 
$T$ a maximal torus, $B$ a Borel subgroup containing $T$, and 
$U$ the unipotent radical of $B$; 
we choose the convention that 
the roots in $B$ are the negative ones. 
Let $X_{*}(T)$ denote the (integral) coweight lattice 
$\Hom(\BC^{*},\,T)$ for $G$, and $X_{*}(T)_{+}$ 
the set of dominant (integral) coweights for $G$; 
we regard the coweight lattice 
$X_{*}(T)$ as an additive subgroup of 
a real form $\Fh_{\BR}:=\BR \otimes_{\BZ} X_{*}(T)$ 
of the Lie algebra $\Fh$ of the maximal torus $T$. 
We denote by $G^{\vee}$ 
the (complex) Langlands dual group of $G$. 

Denote by $\Fg$ the Lie algebra of $G$, 
which is a complex semisimple Lie algebra. 
Let 
\begin{equation*}
\Bigl(A=(a_{ij})_{i,j \in I}, \, 
 \Pi:=\bigl\{\alpha_{j}\bigr\}_{j \in I}, \, 
 \Pi^{\vee}:=\bigl\{h_{j}\bigr\}_{j \in I}, \, 
 \Fh^{\ast},\,\Fh
 \Bigr)
\end{equation*}
be the root datum of $\Fg$, where 
$A=(a_{ij})_{i,j \in I}$ is the Cartan matrix, 
$\Fh$ is the Cartan subalgebra, 
$\Pi:=\bigl\{\alpha_{j}\bigr\}_{j \in I} \subset 
 \Fh^{\ast}:=\Hom_{\BC}(\Fh,\,\BC)$ 
is the set of simple roots, and 
$\Pi^{\vee}:=\bigl\{h_{j}\bigr\}_{j \in I} \subset \Fh$ 
is the set of simple coroots; note that 
$\pair{h_{i}}{\alpha_{j}}=a_{ij}$ for $i,\,j \in I$, 
where $\pair{\cdot}{\cdot}$ denotes the canonical pairing 
between $\Fh$ and $\Fh^{\ast}$, and that $\Fh_{\BR}= 
\sum_{j \in I} \BR h_{j} \subset \Fh$.  
Also, for $h,\,h' \in \Fh_{\BR}$, we write 
$h' \ge h$ if $h'-h \in Q^{\vee}_{+}:=\sum_{j \in I}\BZ_{\ge 0}h_{j}$. 
Let $W:=\langle s_{j} \mid j \in I \rangle$ 
be the Weyl group of $\Fg$, where $s_{j}$, $j \in I$, are 
the simple reflections, with length function 
$\ell:W \rightarrow \BZ_{\ge 0}$, 
the identity element $e \in W$, and 
the longest element $w_{0} \in W$; 
we denote by $\le$ the (strong) Bruhat order on $W$. 
Let $\Fg^{\vee}$ denote the Lie algebra of 
the Langlands dual group $G^{\vee}$ of $G$, 
which is the complex semisimple Lie algebra
associated to the root datum 
\begin{equation*}
\Bigl({}^{t}A=(a_{ji})_{i,j \in I}, \, 
 \Pi^{\vee}=\bigl\{h_{j}\bigr\}_{j \in I}, \, 
 \Pi=\bigl\{\alpha_{j}\bigr\}_{j \in I}, \, 
 \Fh,\,\Fh^{\ast}
 \Bigr);
\end{equation*}
note that the Cartan subalgebra of 
$\Fg^{\vee}$ is $\Fh^{\ast}$, not $\Fh$. 
Let $U_{q}(\Fg^{\vee})$ be 
the quantized universal enveloping algebra of 
$\Fg^{\vee}$ over $\BC(q)$. 
For a dominant coweight 
$\lambda \in X_{\ast}(T)_{+} \subset \Fh_{\BR}$, 
denote by $V(\lambda)$ the irreducible highest 
weight $U_{q}(\Fg^{\vee})$-module 
of highest weight $\lambda$, and by 
$\CB(\lambda)$ the crystal basis of $V(\lambda)$. 

%
\subsection{Mirkovi\'c-Vilonen polytopes.}
\label{subsec:MV}

In this subsection, following \cite{Kam1}, we recall 
a (combinatorial) characterization of 
Mirkovi\'c-Vilonen (MV for short) polytopes; 
the relation between this characterization and 
the original (geometric) definition of MV polytopes 
given by Anderson \cite{A} will be explained 
in \S\ref{subsec:geom}. 

As in \S\ref{subsec:notation}, 
we assume that $\Fg$ is a complex semisimple Lie algebra. 
Let $\mu_{\bullet}=(\mu_{w})_{w \in W}$ be 
a collection of elements of 
$X_{*}(T) \subset \Fh_{\BR}=\sum_{j \in I} \BR h_{j}$.
We call $\mu_{\bullet}=(\mu_{w})_{w \in W}$ a 
Gelfand-Goresky-MacPherson-Serganova 
(GGMS) datum if it satisfies 
the condition that 
%
%
\begin{equation} \label{eq:GGMS}
x^{-1} \cdot \mu_{z} - x^{-1} \cdot \mu_{x} 
 \in Q^{\vee}_{+} \qquad \text{for all $x,\,z \in W$}.
\end{equation}
It follows by induction 
with respect to the (weak) Bruhat order on $W$ 
that $\mu_{\bullet}=(\mu_{w})_{w \in W}$ is a GGMS datum 
if and only if 
%
%
\begin{equation} \label{eq:length}
\mu_{ws_{i}}-\mu_{w} \in \BZ_{\ge 0}\,(w \cdot h_{i}) 
\quad 
\text{for every $w \in W$ and $i \in I$}. 
\end{equation}
%
%
\begin{rem} \label{rem:GGMS-Q}
Let $\mu_{\bullet}=(\mu_{w})_{w \in W}$ be a GGMS datum, 
and take an arbitrary $w \in W$. 
We see from \eqref{eq:GGMS} (with $x=w_{0}$ and $z=w$) that 
$w_{0}^{-1} \cdot \mu_{w} - w_{0}^{-1} \cdot \mu_{w_{0}} \in Q^{\vee}_{+}$, 
which implies that $\mu_{w}$ is contained in $\mu_{w_{0}}-Q^{\vee}_{+}$. 
Since $\mu_{w_{0}} \in X_{\ast}(T)$, we deduce that 
$w^{-1} \cdot \mu_{w}$ is contained in $\mu_{w_{0}}-Q^{\vee}$, 
where we set $Q^{\vee}:=\sum_{j \in I} \BZ h_{j}$. 
Hence
\begin{equation*}
z^{-1} \cdot \mu_{z} - x^{-1} \cdot \mu_{x} \in Q^{\vee} 
\quad \text{for all $x,\,z \in W$}.
\end{equation*}
\end{rem}

%
\begin{rem} \label{rem:GGMS-sum}
Let $\mu_{\bullet}^{(1)}=(\mu_{w}^{(1)})_{w \in W}$ and 
$\mu_{\bullet}^{(2)}=(\mu_{w}^{(2)})_{w \in W}$ be GGMS data. 
Then, it is obvious from the definition of GGMS data 
(i.e., from \eqref{eq:length}) that 
the (componentwise) sum 
\begin{equation*}
\mu_{\bullet}^{(1)}+\mu_{\bullet}^{(2)}:=
 (\mu_{w}^{(1)}+\mu_{w}^{(2)})_{w \in W}
\end{equation*}
of $\mu_{\bullet}^{(1)}$ and $\mu_{\bullet}^{(2)}$ is 
also a GGMS datum. 
\end{rem}

Following \cite{Kam1} and \cite{Kam2}, 
to each GGMS datum $\mu_{\bullet}=(\mu_{w})_{w \in W}$, 
we associate a convex polytope 
$P(\mu_{\bullet}) \subset \Fh_{\BR}$ by:
%
%
\begin{equation} \label{eq:poly}
P(\mu_{\bullet})=
\bigcap_{w \in W}
 \bigl\{ 
 v \in \Fh_{\BR} \mid 
 w^{-1} \cdot v- w^{-1} \cdot \mu_{w} \in 
 \textstyle{\sum_{j \in I}\BR_{\ge 0} h_{j}}
 \bigr\}; 
\end{equation}
the polytope $P(\mu_{\bullet})$ is called 
a pseudo-Weyl polytope with GGMS datum $\mu_{\bullet}$. 
Note that the GGMS datum $\mu_{\bullet}=(\mu_{w})_{w \in W}$ is 
determined uniquely by the convex polytope $P(\mu_{\bullet})$. 
Also, we know from \cite[Proposition~2.2]{Kam1} that 
the set of vertices of the polytope $P(\mu_{\bullet})$ 
is given by the collection $\mu_{\bullet}=(\mu_{w})_{w \in W}$ 
(possibly, with repetitions). In particular, we have
%
%
\begin{equation} \label{eq:conv}
P(\mu_{\bullet}) = 
\Conv\,\bigl\{\mu_{w} \mid w \in W\bigr\},
\end{equation}
where for a subset $X$ of $\Fh_{\BR}$, 
$\Conv X$ denotes the convex hull in $\Fh_{\BR}$ of $X$. 

In the proof of Theorem~\ref{thm:main} below, 
we need the following lemma about 
Minkowski sums of pseudo-Weyl polytopes. 
%
%
\begin{lem}[{\cite[Lemma~6.1]{Kam1}}] \label{lem:Minkowski}
Let $P_{1}=P(\mu_{\bullet}^{(1)})$ and 
$P_{2}=P(\mu_{\bullet}^{(2)})$ be pseudo-Weyl polytopes 
with GGMS data 
$\mu_{\bullet}^{(1)}=(\mu_{w}^{(1)})_{w \in W}$ and 
$\mu_{\bullet}^{(2)}=(\mu_{w}^{(2)})_{w \in W}$, respectively. 
Then, the Minkowski sum
\begin{equation*}
P_{1}+P_{2}:=
 \bigl\{v_{1}+v_{2} \mid v_{1} \in P_{1},\,v_{2} \in P_{2}\bigr\}
\end{equation*}
of the pseudo-Weyl polytopes $P_{1}$ and $P_{2}$ is identical to 
the pseudo-Weyl polytope $P(\mu_{\bullet}^{(1)}+\mu_{\bullet}^{(2)})$ 
having GGMS datum $\mu_{\bullet}^{(1)}+\mu_{\bullet}^{(2)}=
 (\mu_{w}^{(1)}+\mu_{w}^{(2)})_{w \in W}$ 
(see Remark~\ref{rem:GGMS-sum}). 
\end{lem}

Furthermore, we need to 
recall from \cite[\S2.3]{Kam1} the notion of 
Berenstein-Zelevinsky (BZ for short) data. 
We set $\Gamma:=
 \bigl\{w \cdot \Lambda_{j} \mid w \in W,\,j \in I\bigr\}$, 
where $\Lambda_{j}$, $j \in I$, are 
the fundamental weights for $\Fg$. 
Let $P=P(\mu_{\bullet})$ be a pseudo-Weyl polytope 
with GGMS datum $\mu_{\bullet}=(\mu_{w})_{w \in W}$. 
For each $\gamma \in \Gamma$, we set
\begin{equation*}
M_{\gamma}:=\pair{\mu_{w}}{w \cdot \Lambda_{j}}
\quad
\text{if $\gamma=w \cdot \Lambda_{j}$ 
  for some $w \in W$ and $j \in I$};
\end{equation*}
note that the number $M_{\gamma}$ does note depend on 
the expression $\gamma=w \cdot \Lambda_{j}$, 
$w \in W$, $j \in I$, of $\gamma \in \Gamma$.
We call the collection 
$M_{\bullet}=(M_{\gamma})_{\gamma \in \Gamma}$ 
the BZ datum of the pseudo-Weyl polytope $P$. 
We know from \cite[Proposition~2.2]{Kam1} that
%
%
\begin{equation} \label{eq:polybz}
P=P(\mu_{\bullet})=
 \bigl\{v \in \Fh_{\BR} \mid 
   \text{$\pair{v}{\gamma} \ge M_{\gamma}$ for all $\gamma \in \Gamma$}
 \bigr\}.
\end{equation}
Since the proof of \cite[Lemma~4.5.4]{NS-dp} for MV polytopes 
works equally well for pseudo-Weyl polytopes, we have the following. 
%
%
\begin{lem} \label{lem:bz}
Let $P$ and $P'$ be pseudo-Weyl polytopes 
with BZ data $M_{\bullet}=(M_{\gamma})_{\gamma \in \Gamma}$ and 
$M_{\bullet}'=(M_{\gamma}')_{\gamma \in \Gamma}$, respectively. 
Then, $P \subset P'$ if and only if 
$M_{\gamma} \ge M_{\gamma}'$ 
for all $\gamma \in \Gamma$. 
\end{lem}

Let $R(w_{0})$ denote 
the set of all reduced words for $w_{0}$, that is, 
all sequences $(i_{1},\,i_{2},\,\dots,\,i_{m})$ of elements of $I$ 
such that $s_{i_{1}}s_{i_{2}} \cdots s_{i_{m}}=w_{0}$, where 
$m$ is the length $\ell(w_{0})$ of the longest element $w_{0}$.
Let $\bi=(i_{1},\,i_{2},\,\dots,\,i_{m}) \in R(w_{0})$ 
be a reduced word for $w_{0}$. 
We set $\wi{l}:=\si{1}\si{2} \cdots \si{l} \in W$ 
for $0 \le l \le m$. For a GGMS datum 
$\mu_{\bullet}=(\mu_{w})_{w \in W}$, define integers 
(called the lengths of edges) 
$\Ni{l}=\Ni{l}(\mu_{\bullet}) \in \BZ_{\ge 0}$, 
$1 \le l \le m$, via the following 
``length formula'' (see \cite[Eq.\,(8)]{Kam1} and 
\eqref{eq:length} above): 
%
%
\begin{equation} \label{eq:n}
\mu_{\wi{l}}-\mu_{\wi{l-1}}=\Ni{l}\,\wi{l-1} \cdot h_{i_{l}}.
\end{equation}

\vsp

{\small 
\hspace{55mm}
\input{figure-tensor003.tex}
}

\vsp

Now we are ready to give a (combinatorial) characterization of 
Mirkovi\'c-Vilonen (MV) polytopes, due to Kamnitzer \cite{Kam1}. 
This result holds for an arbitrary complex semisimple Lie algebra $\Fg$, 
but we give its precise statement only in the case that 
$\Fg$ is simply-laced since we do not make use of it 
in this paper; we merely mention that 
when $\Fg$ is not simply-laced, 
there are also conditions on the lengths 
$\Ni{l}$, $1 \le l \le m$, $\bi \in R(w_{0})$, 
for the other possible values of $a_{ij}$ and $a_{ji}$
(we refer the reader to \cite[\S3]{BeZe} for explicit formulas).
%
%
\begin{dfn} \label{dfn:MV}
A GGMS datum $\mu_{\bullet}=(\mu_{w})_{w \in W}$ is 
said to be a Mirkovi\'c-Vilonen (MV) datum if 
it satisfies the following conditions: 

(1) If $\bi=(i_{1},\,i_{2},\,\dots,\,i_{m}) \in R(w_{0})$ and 
$\bj=(j_{1},\,j_{2},\,\dots,\,j_{m}) \in R(w_{0})$ are related by 
a $2$-move, that is, if there exist indices $i,\,j \in I$ 
with $a_{ij}=a_{ji}=0$ and an integer $0 \le k \le m-2$ 
such that $i_{l}=j_{l}$ for all $1 \le l \le m$ with 
$l \ne k+1,\,k+2$, and such that 
$i_{k+1}=j_{k+2}=i$, $i_{k+2}=j_{k+1}=j$, then 
there hold
\begin{equation*}
\begin{cases}
\Ni{l}=\Nj{l} \quad 
  \text{for all $1 \le l \le m$ with $l \ne k+1,\,k+2$, and} \\[1.5mm]
\Ni{k+1}=\Nj{k+2}, \quad \Ni{k+2}=\Nj{k+1}.
\end{cases}
\end{equation*}

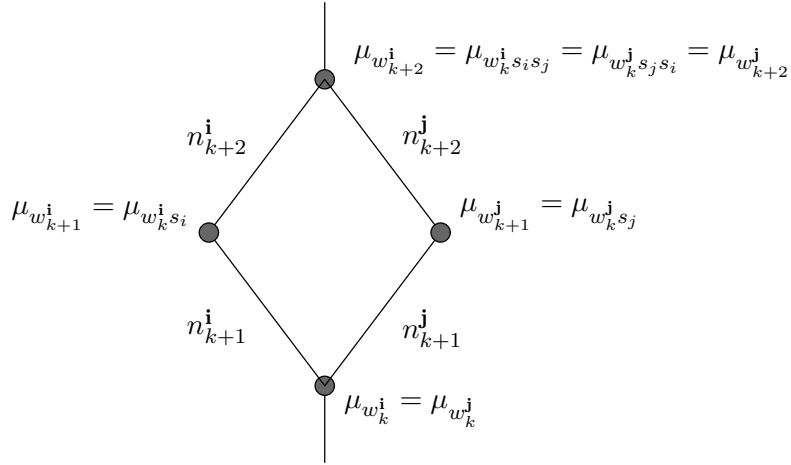
\vspace{5mm}

\newcommand{\vertexa}{
  $\mu_{\wi{k+2}}=\mu_{\wi{k}s_is_j}=
   \mu_{\wj{k}s_js_i}=\mu_{\wj{k+2}}$
}

{\small 
\hspace*{-20mm}
\input{figure-tensor001.tex}
}

\vsp

(2) If $\bi=(i_{1},\,i_{2},\,\dots,\,i_{m}) \in R(w_{0})$ and 
$\bj=(j_{1},\,j_{2},\,\dots,\,j_{m}) \in R(w_{0})$ are 
related by a $3$-move, that is, 
if there exist indices $i,\,j \in I$ 
with $a_{ij}=a_{ji}=-1$ and an integer $0 \le k \le m-3$ 
such that $i_{l}=j_{l}$ for all $1 \le l \le m$ with 
$l \ne k+1,\,k+2,\,k+3$, and such that
$i_{k+1}=i_{k+3}=j_{k+2}=i$, 
$i_{k+2}=j_{k+1}=j_{k+3}=j$, then 
there hold 
\begin{equation*}
\begin{cases}
\Ni{l}=\Nj{l} \quad
  \text{for all $1 \le l \le m$ with $l \ne k+1,\,k+2,\,k+3$, and} \\[1.5mm]
\Nj{k+1}=\Ni{k+2}+\Ni{k+3}-\min \bigl(\Ni{k+1},\, \Ni{k+3}\bigr), \\[1.5mm]
\Nj{k+2}=\min \bigl(\Ni{k+1},\, \Ni{k+3}\bigr), \\[1.5mm]
\Nj{k+3}=\Ni{k+1}+\Ni{k+2}-\min \bigl(\Ni{k+1},\, \Ni{k+3}\bigr).
\end{cases}
\end{equation*}

\vspace{5mm}

\newcommand{\vertexb}{
  $\mu_{\wi{k+3}}=\mu_{\wi{k}s_is_js_i}=
   \mu_{\wj{k}s_js_is_j}=\mu_{\wj{k+3}}$
}

{\small 
\hspace*{-22.5mm}
\input{figure-tensor002.tex}
}
\end{dfn}

The pseudo-Weyl polytope $P(\mu_{\bullet})$ 
with GGMS datum $\mu_{\bullet}=(\mu_{w})_{w \in W}$ 
(see \eqref{eq:poly}) is 
a Mirkovi\'c-Vilonen (MV) polytope if and only if 
the GGMS datum $\mu_{\bullet}=(\mu_{w})_{w \in W}$ is 
an MV datum (see the proof of \cite[Proposition~5.4]{Kam1} and 
the comment following \cite[Theorem~7.1]{Kam1}).
Also, for a dominant coweight 
$\lambda \in X_{*}(T)_{+} \subset \Fh_{\BR}$ and a coweight 
$\nu \in X_{*}(T) \subset \Fh_{\BR}$, 
an MV polytope $P=P(\mu_{\bullet})$ with GGMS datum 
$\mu_{\bullet}=(\mu_{w})_{w \in W}$ is 
an MV polytope of highest vertex $\lambda$ and lowest vertex $\nu$ 
if and only if $\mu_{w_{0}}=\lambda$, $\mu_{e}=\nu$, and 
$P$ is contained in the convex hull $\Conv(W \cdot \lambda)$ of 
the $W$-orbit $W \cdot \lambda \subset \Fh_{\BR}$ 
(see \cite[Proposition~7]{A}); 
we denote by $\mv(\lambda)_{\nu}$ the set of MV polytopes 
of highest vertex $\lambda$ and lowest vertex $\nu$. 
For each dominant coweight 
$\lambda \in X_{*}(T)_{+} \subset \Fh_{\BR}$, we set
\begin{equation*}
\mv(\lambda):=\bigsqcup_{
 \nu \in X_{*}(T)} \mv(\lambda)_{\nu}.
\end{equation*}

%
\subsection{Relation between MV polytopes and MV cycles.}
\label{subsec:geom}

In this subsection, we review the relation of MV polytopes with 
MV cycles in the affine Grassmannian. 

Let us recall the definition of MV cycles 
in the affine Grassmannian, following 
\cite{MV1}, \cite{MV2} (and \cite{A}). 
Let $G$ be a complex connected semisimple algebraic group 
with Lie algebra $\Fg$, as in \S\ref{subsec:notation}. 
Let $\CO = \BC[[t]]$ denote the ring of formal power series, 
and $\CK = \BC((t))$ the field of formal Laurent series 
(the fraction field of $\CO$). 
The affine Grassmannian $\Gr$ for $G$ over $\BC$ is 
defined to be the quotient $G(\CK)/G(\CO)$, 
equipped with the structure of a complex algebraic ind-scheme, 
where $G(\CK)$ denotes the set of $\CK$-valued points of $G$, and 
$G(\CO) \subset G(\CK)$ denotes the set of $\CO$-valued points of $G$; 
we denote by $\pi:G(\CK) \twoheadrightarrow \Gr=G(\CK)/G(\CO)$ 
the natural quotient map, which is locally trivial 
in the Zariski topology. In what follows, 
for a subgroup $H \subset G(\CK)$ 
that is stable under the adjoint action of $T$ and 
for an element $w$ of the Weyl group $W \cong N_{G}(T)/T$ of $G$, 
we denote by ${}^{w}H$ the $w$-conjugate $\dot{w}H\dot{w}^{-1}$ of $H$, 
where $\dot{w} \in N_{G}(T)$ is a lift of $w \in W$. 

Since each coweight $\nu \in X_{*}(T)=\Hom(\BC^{\ast},\,T)$ 
is a regular map from $\BC^{\ast}$ to $T \subset G$, 
it gives a point $t^{\nu} \in G (\CK)$, 
which in turn, descends to a point 
$[t^{\nu}] \in \Gr=G(\CK)/G(\CO)$. 

For each $\nu \in X_{*}(T)$, we set
\begin{equation*}
\Gr^{\nu}:=G(\CO)[t^{\nu}] \subset \Gr,
\end{equation*}
the $G(\CO)$-orbit of $[t^{\nu}]$, which is a smooth 
quasi-projective algebraic variety over $\BC$. 
Also, for each $\nu \in X_{*}(T)$ and $w \in W$, 
we set 
\begin{equation*}
S_{\nu}^{w} := {}^{w}U(\CK)[t^{\nu}] \subset \Gr,
\end{equation*}
the ${}^{w}U(\CK)$-orbit of $[t^{\nu}]$, which is a (locally closed) 
ind-subscheme of $\Gr$; we write simply $S_{\nu}$ for $S_{\nu}^{e}$.
Then, we know the following two kinds 
of decompositions of $\Gr$ into orbits. 
First, we have
\begin{equation*}
\Gr=\bigsqcup_{\lambda \in X_{*}(T)_{+}} \Gr^{\lambda} \qquad 
\text{(Cartan decomposition)},
\end{equation*}
with $\Gr^{w \cdot \lambda}=\Gr^{\lambda}$ 
for $\lambda \in X_{*}(T)_{+}$ and $w \in W$; 
note that (see, for example, \cite[\S2]{MV2}) 
for each $\lambda \in X_{*}(T)_{+}$, the quasi-projective algebraic
variety $\Gr^{\lambda}$ is simply-connected, and of dimension 
$2\pair{\lambda}{\rho}$, where $\rho$ denotes half the sum
of the positive roots $\alpha \in \Delta_{+}$ for $G$, i.e., 
$2\rho=\sum_{\alpha \in \Delta_{+}}\alpha$. 
Second, we have for each $w \in W$, 
\begin{equation*}
\Gr=\bigsqcup_{\nu \in X_{*}(T)} S_{\nu}^{w} \qquad
\text{(Iwasawa decomposition)}.
\end{equation*}
Moreover, the (Zariski) closure relations among these orbits are 
described as follows (see \cite[\S2 and \S3]{MV2}): 
%
%
\begin{equation} \label{eq:Grlam}
\ol{\Gr^{\lambda}}=
 \bigsqcup_{
   \begin{subarray}{c} 
   \lambda' \in X_{\ast}(T)_{+} \\[1mm]
   \lambda' \le \lambda
   \end{subarray}
 } \Gr^{\lambda'}
\qquad \text{for $\lambda \in X_{*}(T)_{+}$};
\end{equation}
%
%
\begin{equation} \label{eq:Snuw}
\ol{S_{\nu}^{w}}=
 \bigsqcup_{
   \begin{subarray}{c} 
   \gamma \in X_{\ast}(T) \\[1mm]
   w^{-1} \cdot \gamma \ge w^{-1} \cdot \nu
   \end{subarray}
 } S_{\gamma}^{w}
\qquad \text{for $\nu \in X_{\ast}(T)$ and $w \in W$}.
\end{equation}

For $\lambda \in X_{*}(T)_{+}$, 
let $L(\lambda)$ denote the irreducible finite-dimensional 
representation of the Langlands dual 
group $G^{\vee}$ of $G$ of highest weight $\lambda$, and 
$\Omega(\lambda) \subset X_{*}(T)$ 
the set of weights of $L(\lambda)$. 
We know from \cite[Theorem~3.2 and Remark~3.3]{MV2} that 
$\nu \in X_{*}(T)$ is an element of $\Omega(\lambda)$ 
if and only if $\Gr^{\lambda} \cap S_{\nu} \ne \emptyset$, 
and, in this case, the intersection $\Gr^{\lambda} \cap S_{\nu}$ 
is of pure dimension $\pair{\lambda-\nu}{\rho}$. 

Now we come to the definition of MV cycles 
in the affine Grassmannian.
%
%
\begin{dfn}[{\cite[\S3]{MV2}; see also \cite[\S5.3]{A}}] \label{dfn:MVcycle}
Let $\lambda \in X_{*}(T)_{+}$ and $\nu \in X_{*}(T)$ be such that 
$\Gr^{\lambda} \cap S_{\nu} \ne \emptyset$, 
i.e., $\nu \in \Omega(\lambda)$. 
An MV cycle of highest weight $\lambda$ and weight $\nu$ is 
defined to be an irreducible component of the (Zariski) closure of 
the intersection $\Gr^{\lambda} \cap S_{\nu}$. 
\end{dfn}

We denote by $\CZ(\lambda)_{\nu}$ the set of MV cycles of 
highest weight $\lambda \in X_{*}(T)_{+}$ and weight $\nu \in X_{*}(T)$. 
Also, for each $\lambda \in X_{*}(T)_{+}$, we set 
\begin{equation*}
\CZ(\lambda) := 
 \bigsqcup_{\nu \in X_{*}(T)} \CZ(\lambda)_{\nu},
\end{equation*}
where $\CZ(\lambda)_{\nu} := \emptyset$ if 
$\Gr^{\lambda} \cap S_{\nu}=\emptyset$. 

Motivated by the discovery of MV cycles in the affine Grassmannian, 
Anderson \cite{A} proposed considering the ``moment map images'' 
of MV cycles as follows: Let $\lambda \in X_{*}(T)_{+}$. 
For an MV cycle $\bb \in \CZ(\lambda)$, we set 
\begin{equation*}
P(\bb):=\Conv \bigl\{
  \nu \in X_{\ast}(T) \subset \Fh_{\BR} \mid 
  [t^{\nu}] \in \bb \bigr\},
\end{equation*}
and call $P(\bb) \subset \Fh_{\BR}$ the moment map 
image of $\bb$\,; note that $P(\bb)$ is indeed 
a convex polytope in $\Fh_{\BR}$. 

The following theorem, due to Kamnitzer \cite{Kam1}, 
establishes an explicit relationship between 
MV polytopes and MV cycles. 
%
%
\begin{thm} \label{thm:Kam1}
{\rm (1)} 
Let $\lambda \in X_{*}(T)_{+}$ and $\nu \in X_{*}(T)$ be 
such that $\Gr^{\lambda} \cap S_{\nu} \ne \emptyset$. 
If $\mu_{\bullet}=(\mu_{w})_{w \in W}$ denotes the GGMS datum 
of an MV polytope $P \in \mv(\lambda)_{\nu}$, that is, 
$P=P(\mu_{\bullet}) \in \mv(\lambda)_{\nu}$, then 
\begin{equation*}
\bb(\mu_{\bullet}):=
\ol{ \bigcap_{w \in W} S^{w}_{\mu_{w}} } \subset \ol{\Gr^{\lambda}}
\end{equation*}
is an MV cycle that belongs to $\CZ(\lambda)_{\nu}$. 

{\rm (2)} 
Let $\lambda \in X_{*}(T)_{+}$. 
For an MV polytope $P=P(\mu_{\bullet}) \in \mv(\lambda)$ 
with GGMS datum $\mu_{\bullet}$, 
we set $\Phi_{\lambda}(P):=\bb(\mu_{\bullet})$. Then, 
the map $\Phi_{\lambda}:\mv(\lambda) \rightarrow \CZ(\lambda)$, 
$P \mapsto \Phi_{\lambda}(P)$, is a bijection 
from $\mv(\lambda)$ onto $\CZ(\lambda)$ 
such that $\Phi_{\lambda}(\mv(\lambda)_{\nu})=
\CZ(\lambda)_{\nu}$ for all $\nu \in X_{\ast}(T)$
with $\Gr^{\lambda} \cap S_{\nu} \ne \emptyset$. 
In particular, for each MV cycle $\bb \in \CZ(\lambda)$, 
there exists a unique MV datum $\mu_{\bullet}$ 
such that $\bb=\bb(\mu_{\bullet})$, and in this case, 
the moment map image $P(\bb)$ of the MV cycle $\bb=\bb(\mu_{\bullet})$
is identical to the MV polytope $P(\mu_{\bullet}) \in \mv(\lambda)$. 
\end{thm}
%
%
\begin{rem}[{\cite[\S2.2]{Kam1}}] \label{rem:moment}
For $\nu \in X_{*}(T)$ and $w \in W$, 
the ``moment map image'' $P(\ol{S_{\nu}^{w}})$ of 
$\ol{S_{\nu}^{w}}$ is, by definition, 
the convex hull in $\Fh_{\BR}$ of the set
$\bigl\{
  \gamma \in X_{*}(T) \subset \Fh_{\BR} \mid 
  [t^{\gamma}] \in \ol{S_{\nu}^{w}}
\bigr\} \subset \Fh_{\BR}$,
which is identical to the (shifted) convex cone
$\bigl\{
 v \in \Fh_{\BR} \mid 
 w^{-1} \cdot v - w^{-1} \cdot \nu \in 
 \textstyle{\sum_{j \in I}\BR_{\ge 0}h_{j}}
\bigr\}$. 
\end{rem}

%
\subsection{Lusztig-Berenstein-Zelevinsky crystal structure.}
\label{subsec:MVcry}

We keep the notation and assumptions of \S\ref{subsec:MV}.
For an MV datum $\mu_{\bullet}=(\mu_{w})_{w \in W}$ and $j \in I$, 
we denote by $f_{j}\mu_{\bullet}$ 
(resp., $e_{j}\mu_{\bullet}$ if $\mu_{e} \ne \mu_{s_{j}}$; 
note that $\mu_{s_{j}}-\mu_{e} \in 
\BZ_{\ge 0}h_{j}$ by \eqref{eq:length})
a unique MV datum $\mu_{\bullet}'=(\mu_{w}')_{w \in W}$ 
such that $\mu_{e}'=\mu_{e}-h_{j}$ 
(resp., $\mu_{e}'=\mu_{e}+h_{j}$) and 
$\mu_{w}'=\mu_{w}$ for all $w \in W$ with $s_{j}w < w$ 
(see \cite[Theorem~3.5]{Kam2} and its proof); 
note that $\mu_{w_{0}}'=\mu_{w_{0}}$ and 
$\mu_{s_j}'=\mu_{s_j}$. 

Let $\lambda \in X_{\ast}(T)_{+} \subset \Fh_{\BR}$ 
be a dominant coweight. 
Following \cite[\S6.2]{Kam2}, we endow $\mv(\lambda)$ 
with the Lusztig-Berenstein-Zelevinsky (LBZ for short) 
crystal structure for $U_{q}(\Fg^{\vee})$
as follows. 
Let $P=P(\mu_{\bullet}) \in \mv(\lambda)$ be 
an MV polytope with GGMS datum 
$\mu_{\bullet}=(\mu_{w})_{w \in W}$.
The weight $\wt(P)$ of $P$ is, by definition, 
equal to the vertex $\mu_{e} \in \lambda-Q^{\vee}_{+}$. 
For each $j \in I$, 
we define the lowering Kashiwara operator
$f_{j}:\mv(\lambda) \cup \{\bzero\} \rightarrow 
 \mv(\lambda) \cup \{\bzero\}$ and 
the raising Kashiwara operator
$e_{j}:\mv(\lambda) \cup \{\bzero\} \rightarrow 
 \mv(\lambda) \cup \{\bzero\}$ by: 
\begin{align*}
e_{j}\bzero=f_{j}\bzero & :=\bzero, \\[3mm]
f_{j}P=f_{j}P(\mu_{\bullet}) & :=
\begin{cases}
 P(f_{j}\mu_{\bullet}) 
 & \text{if $P(f_{j}\mu_{\bullet}) \subset 
         \Conv (W \cdot \lambda)$}, \\[1.5mm]
 \bzero & \text{otherwise}, 
 \end{cases} \\[3mm]
e_{j}P=e_{j}P(\mu_{\bullet}) & :=
\begin{cases}
 P(e_{j}\mu_{\bullet}) & 
 \text{if $\mu_{e} \ne \mu_{s_{j}}$ 
 (i.e., $\mu_{s_{j}}-\mu_{e} \in \BZ_{> 0}h_{j}$)}, \\[1.5mm]
 \bzero & \text{otherwise}, 
 \end{cases}
\end{align*}
where $\bzero$ is an additional element, 
not contained in $\mv(\lambda)$. 
For $j \in I$, we set 
$\ve_{j}(P):=
 \max \bigl\{k \in \BZ_{\ge 0} \mid e_{j}^{k}P \ne \bzero \bigr\}$ and 
$\vp_{j}(P):=
 \max \bigl\{k \in \BZ_{\ge 0} \mid f_{j}^{k}P \ne \bzero \bigr\}$.

%
\begin{thm}[{\cite[Theorem~6.4]{Kam2}}] \label{thm:kam-int}
The set $\mv(\lambda)$, equipped with the maps 
$\wt$, $e_{j},\,f_{j} \ (j \in I)$, and 
$\ve_{j},\,\vp_{j} \ (j \in I)$ above, is 
a crystal for $U_{q}(\Fg^{\vee})$. 
Moreover, there exists a unique isomorphism 
$\Psi_{\lambda}:\CB(\lambda) 
 \stackrel{\sim}{\rightarrow} 
 \mv(\lambda)$ of crystals for $U_{q}(\Fg^{\vee})$. 
\end{thm}
%
%
\begin{rem} \label{rem:LBZ=BFG}
Kamnitzer \cite[Theorem~4.7]{Kam2} proved that 
for each $\lambda \in X_{*}(T)_{+}$, the bijection 
$\Phi_{\lambda}:\mv(\lambda) \rightarrow \CZ(\lambda)$ in 
Theorem~\ref{thm:Kam1}\,(2) intertwines the LBZ crystal 
structure on $\mv(\lambda)$ and the crystal structure on $\CZ(\lambda)$ 
defined in \cite{BrGa} (and \cite{BFG}). 
\end{rem}

For $P \in \mv(\lambda)$ and $j \in I$, we set
\begin{equation*}
f_{j}^{\max}P:=f_{j}^{\vp_{j}(P)}P. 
\end{equation*}

%
\subsection{Description of MV polytopes in terms of Kashiwara data.}
\label{subsec:ehrig}

The following description of MV polytopes, 
due to Ehrig, is obtained as a corollary of 
his main result \cite[Theorem~1.1]{E}.
%
%
\begin{prop} \label{prop:ehrig}
Let $\lambda \in X_{\ast}(T)_{+} \subset \Fh_{\BR}$, 
and $P=P(\mu_{\bullet}) \in \mv(\lambda)$ 
an MV polytope with GGMS datum $\mu_{\bullet}=(\mu_{w})_{w \in W}$. 
Let $x \in W$, and $x=s_{i_{1}}s_{i_{2}} \cdots s_{i_{k}}$ 
an arbitrary reduced expression of $x$. Then, we have 
%
%
\begin{equation} \label{eq:ehrig}
\mu_{x}=x \cdot 
 \wt (f_{i_{k}}^{\max} \cdots f_{i_{2}}^{\max}f_{i_{1}}^{\max}P).
\end{equation}
\end{prop}

Because this result itself follows easily from 
Kamnitzer's result \cite[Theorem~6.6]{Kam2} 
describing Kashiwara data in terms of BZ data, 
we include its short proof for the convenience of the reader; 
in fact, since the reverse implication is
shown in the proof of \cite[Corollary~7.5]{E2}, 
these two results are indeed equivalent. 

\begin{proof}[Proof of Proposition~\ref{prop:ehrig}]
Take $i_{k+1},\,i_{k+2},\,\dots,\,i_{m} \in I$ in such a way that 
\begin{equation*}
\bi=(i_{1},\,i_{2},\,\dots,\,i_{k},\,i_{k+1},\,i_{k+2},\,\dots,\,i_{m})
\end{equation*}
is a reduced word for the longest element $w_{0} \in W$, i.e., 
$\bi \in R(w_{0})$. We define a sequence 
$(p_{1},\,p_{2},\,\dots,\,p_{m}) \in \BZ_{\ge 0}^{m}$ 
of nonnegative integers (called the $\bi$-Kashiwara datum of $P$) by:
%
%
\begin{equation} \label{eq:ikas}
p_{l}:=\vp_{i_{l}}(
  f_{i_{l-1}}^{\max} \cdots f_{i_{2}}^{\max}f_{i_{1}}^{\max}P)
\quad
\text{for $1 \le l \le m$}.
\end{equation}
Then we know from \cite[Theorem~6.6]{Kam2} that 
%
%
\begin{equation} \label{eq:kam-kas}
p_{l}=
 M_{\wi{l-1} \cdot \Lambda_{i_{l}}}-
 M_{\wi{l} \cdot \Lambda_{i_{l}}}
\quad
\text{for $1 \le l \le m$},
\end{equation}
where $M_{\bullet}=(M_{\gamma})_{\gamma \in \Gamma}$ denotes 
the BZ datum of $P=P(\mu_{\bullet}) \in \mv(\lambda)$. 

Because the $x \cdot \Lambda_{j}$, $j \in I$, form a basis of $\Fh^{\ast}$, 
in order to show the equation \eqref{eq:ehrig}, 
it suffices to prove that 
%
%
\begin{equation} \label{eq:ehrig1}
\Bpair{x \cdot 
 \wt (f_{i_{k}}^{\max} \cdots f_{i_{2}}^{\max}f_{i_{1}}^{\max}P)}
{x \cdot \Lambda_{j}}=
\pair{\mu_{x}}{x \cdot \Lambda_{j}}
\quad
\text{for all $j \in I$}. 
\end{equation}
Fix $j \in I$. By definition, the right-hand side of \eqref{eq:ehrig1} is 
equal to $M_{x \cdot \Lambda_{j}}$. 
Also, we see from the definition \eqref{eq:ikas} of 
the $\bi$-Kashiwara datum and the equation \eqref{eq:kam-kas} that 
\begin{align*}
& 
\Bpair{x \cdot \wt (f_{i_{k}}^{\max} \cdots f_{i_{2}}^{\max}f_{i_{1}}^{\max}P)}
{x \cdot \Lambda_{j}}=
\Bpair{\wt (f_{i_{k}}^{\max} \cdots f_{i_{2}}^{\max}f_{i_{1}}^{\max}P)}
{\Lambda_{j}} \\[3mm]
& \qquad =
\Bpair{\wt (f_{i_{k}}^{p_{k}} \cdots f_{i_{2}}^{p_{2}}f_{i_{1}}^{p_{1}}P)}
{\Lambda_{j}} \qquad \text{by \eqref{eq:ikas}} \\[5mm]
& 
\qquad 
=
\left\langle 
 \wt P - \sum_{l=1}^{k}p_{l} h_{i_{l}}, \, \Lambda_{j}
\right\rangle \\[3mm]
& 
\qquad
= 
\left\langle 
\wt P - \sum_{l=1}^{k} 
(M_{\wi{l-1} \cdot \Lambda_{i_{l}}}-
 M_{\wi{l} \cdot \Lambda_{i_{l}}}) h_{i_{l}}, \, \Lambda_{j}
\right\rangle 
\quad \text{by \eqref{eq:kam-kas}} \\[3mm]
& \qquad = \pair{\wt P}{\Lambda_{j}} -
\sum_{l=1}^{k} 
(M_{\wi{l-1} \cdot \Lambda_{i_{l}}}-
 M_{\wi{l} \cdot \Lambda_{i_{l}}})\delta_{i_{l},\,j}.
\end{align*}
Since $\pair{\wt P}{\Lambda_{j}}=\pair{\mu_{e}}{\Lambda_{j}}=M_{\Lambda_{j}}$ 
by the definition of $\wt P$, the left-hand side of \eqref{eq:ehrig1} is equal to 
\begin{equation*}
M_{\Lambda_{j}} -
\sum_{l=1}^{k} 
(M_{\wi{l-1} \cdot \Lambda_{i_{l}}}-
 M_{\wi{l} \cdot \Lambda_{i_{l}}})\delta_{i_{l},j}.
\end{equation*}
If we write the set 
$\bigl\{1 \le l \le k \mid i_{l}=j\bigr\}$ as: 
$\bigl\{a_{1} < a_{2} < \cdots < a_{s}\bigr\}$, 
then we have
\begin{equation*}
M_{\Lambda_{j}} -
\sum_{l=1}^{k} 
(M_{\wi{l-1} \cdot \Lambda_{i_{l}}}-
 M_{\wi{l} \cdot \Lambda_{i_{l}}})\delta_{i_{l},j}
=
M_{\Lambda_{j}}-
\sum_{t=1}^{s} 
(M_{\wi{a_{t}-1} \cdot \Lambda_{j}}-M_{\wi{a_{t}} \cdot \Lambda_{j}}).
\end{equation*}
Here, observe that 

i) $M_{\wi{a_{1}-1} \cdot \Lambda_{j}}=M_{\Lambda_{j}}$ 
since $i_{l} \ne j$ for $1 \le l \le a_{1}-1$; 

ii) for each $1 \le t \le s-1$, 
$M_{\wi{a_{t+1}-1} \cdot \Lambda_{j}}=M_{\wi{a_{t}} \cdot \Lambda_{j}}$ 
since $i_{l} \ne j$ for $a_{t}+1 \le l \le a_{t+1}-1$; 

iii) $M_{\wi{a_{s}} \cdot \Lambda_{j} }=
 M_{\wi{k} \cdot \Lambda_{j} }$ since 
$i_{l} \ne j$ for $a_{s}+1 \le l \le k$. 

\noindent 
Therefore, we see that 
the left-hand side of \eqref{eq:ehrig1} is equal to 
\begin{align*}
& M_{\Lambda_{j}}-
  \sum_{t=1}^{s} 
  (M_{\wi{a_{t}-1} \cdot \Lambda_{j}}-M_{\wi{a_{t}} \cdot \Lambda_{j}}) \\
& = M_{\Lambda_{j}}-\left\{
  (M_{\wi{a_{1}-1} \cdot \Lambda_{j}}-M_{\wi{a_{1}} \cdot \Lambda_{j}})+
  \sum_{t=2}^{s-1} 
  (M_{\wi{a_{t}-1} \cdot \Lambda_{j}}-M_{\wi{a_{t}} \cdot \Lambda_{j}})+
  (M_{\wi{a_{s}-1} \cdot \Lambda_{j}}-M_{\wi{a_{s}} \cdot \Lambda_{j}})
  \right\} \\[3mm]
& = M_{\Lambda_{j}}-\left\{
  (M_{\Lambda_{j}}-M_{\wi{a_{1}} \cdot \Lambda_{j}})+
  \sum_{t=2}^{s-1} 
  (M_{\wi{a_{t-1}} \cdot \Lambda_{j}}-M_{\wi{a_{t}} \cdot \Lambda_{j}})+
  (M_{\wi{a_{s-1}} \cdot \Lambda_{j}}-M_{\wi{k} \cdot \Lambda_{j}})
  \right\} \\[3mm]
& = M_{\wi{k} \cdot \Lambda_{j}} 
  = M_{x \cdot \Lambda_{j}} 
  \quad \text{since $\wi{k}=s_{i_{1}}s_{i_{2}} \cdots s_{i_{k}}=x$},
\end{align*}
as desired. This proves the proposition.
\end{proof}

%
\section{Tensor products and Minkowski sums of MV polytopes.}
\label{sec:main}

%
\subsection{Main result.}
\label{subsec:main}

Let $\lambda_{1},\,\lambda_{2} \in X_{\ast}(T)_{+} \subset \Fh_{\BR}$ 
be dominant coweights. 
Since $\mv(\lambda) \cong \CB(\lambda)$ as crystals 
for every dominant coweight $\lambda \in X_{\ast}(T)_{+} \subset \Fh_{\BR}$, 
the tensor product 
$\mv(\lambda_{1}) \otimes \mv(\lambda_{2})$ of 
the crystals $\mv(\lambda_{1})$ and $\mv(\lambda_{2})$ 
decomposes into a disjoint union of 
connected components as follows:
\begin{equation*}
\mv(\lambda_{1}) \otimes \mv(\lambda_{2}) \cong 
 \bigoplus_{
   \lambda \in X_{\ast}(T)_{+}
   } 
\mv(\lambda)^{\oplus m_{\lambda_{1},\lambda_{2}}^{\lambda}},
\end{equation*}
where $m_{\lambda_{1},\lambda_{2}}^{\lambda} \in \BZ_{\ge 0}$ denotes 
the multiplicity of $\mv(\lambda)$ in 
$\mv(\lambda_{1}) \otimes \mv(\lambda_{2})$. 
For each dominant coweight 
$\lambda \in X_{\ast}(T)_{+} \subset \Fh_{\BR}$ 
such that $m_{\lambda_{1},\lambda_{2}}^{\lambda} \ge 1$, 
we take (and fix) an arbitrary embedding 
$\iota_{\lambda}:\mv(\lambda) \hookrightarrow 
 \mv(\lambda_{1}) \otimes \mv(\lambda_{2})$ 
of crystals that maps $\mv(\lambda)$ 
onto a connected component of 
$\mv(\lambda_{1}) \otimes \mv(\lambda_{2})$, 
which is isomorphic to $\mv(\lambda)$ as a crystal. 

The following theorem is the main result of this paper; 
in our previous paper \cite{KNS}, we proved the same assertion 
under the assumption that $P_{1} \in \mv(\lambda_{1})$ is an 
extremal MV polytope. 
%
%
\begin{thm} \label{thm:main}
Keep the notation above.
Let $P_{1} \in \mv(\lambda_{1})$, 
$P_{2} \in \mv(\lambda_{2})$, and $P \in \mv(\lambda)$ 
be such that $\iota_{\lambda}(P)=P_{1} \otimes P_{2}$ 
for some dominant coweight $\lambda \in 
X_{\ast}(T)_{+} \subset \Fh_{\BR}$. 
Then, $P$ is contained in the Minkowski sum $P_{1}+P_{2}$ of 
the MV polytopes $P_{1} \in \mv(\lambda_{1})$ and 
$P_{2} \in \mv(\lambda_{2})$. Namely, we have the inclusion 
\begin{equation*}
P \subset P_{1} + P_{2}.
\end{equation*}
\end{thm}

%
\subsection{A key inequality and its application.}
\label{subsec:keys}
Let $\lambda \in X_{\ast}(T)_{+} \subset \Fh_{\BR}$ 
be a fixed (but arbitrary) dominant coweight, and 
$P=P(\mu_{\bullet}) \in \mv(\lambda)$ an (arbitrary) MV polytope 
with GGMS datum $\mu_{\bullet}=(\mu_{w})_{w \in W}$.

In view of the agreement of the LBZ and BFG crystal structures
on the set of MV polytopes (see Remark~\ref{rem:LBZ=BFG}), 
we deduce the following fact from \cite[Proposition~4.2]{BG}; 
note the convention in \cite{BG} that the roots in $B$ are 
the positive ones, which is opposite to ours. 

%
\begin{fact} \label{fact:fjP}
Keep the setting above. Let $j \in I$, and 
assume that $f_{j}P \ne \bzero$. Then we have 
$f_{j}P \supset P$. 
\end{fact}

By combining this fact with Proposition~\ref{prop:ehrig}, 
we can prove the following inequality, which plays a key role 
in the proof of Theorem~\ref{thm:main}; in the Appendix, 
we will give a purely geometric proof of this inequality. 
%
%
\begin{prop} \label{prop:key1}
With the notation as above, 
let $x,\,z \in W$ be such that $z \le x$ 
in the Bruhat order on $W$. Then, we have 
%
%
\begin{equation} \label{eq:key1}
z^{-1} \cdot \mu_{z} \ge x^{-1} \cdot \mu_{x}.
\end{equation}
\end{prop}
%
%
\begin{rem} \label{rem:key1}
Keep the notation and assumptions in Proposition~\ref{prop:key1}. 
It follows from Remark~\ref{rem:GGMS-Q} and 
the definition of the order $\ge$ on $\Fh_{\BR}$ that
\begin{align*}
z^{-1} \cdot \mu_{z} \ge x^{-1} \cdot \mu_{x} 
 & \quad \Leftrightarrow \quad
\pair{z^{-1} \cdot \mu_{z}}{\Lambda_{j}} \ge 
\pair{x^{-1} \cdot \mu_{x}}{\Lambda_{j}}
\quad \text{for all $j \in I$}. \\
& \quad \Leftrightarrow \quad
M_{z \cdot \Lambda_{j}} \ge M_{x \cdot \Lambda_{j}}
\quad \text{for all $j \in I$}.
\end{align*}
\end{rem}
%
%
\begin{rem} \label{rem:key1a}
It is well-known (see, for example, \cite[Theorem~2.6.1]{BB}) that 
for $x,\,z \in W$, $z \le x$ in the Bruhat order on $W$ if and only if 
$zW_{\Lambda_{j}} \le xW_{\Lambda_{j}}$ in the Bruhat order on the cosets
$W/W_{\Lambda_{j}}$ modulo the stabilizer $W_{\Lambda_{j}}$ of $\Lambda_{j}$ in $W$ 
for all $j \in I$. 
Moreover, if we are in the case of type $A$, then it is well-known 
(see, for example, \cite[\S3.2]{FZ}) that for each (fixed) $j \in I$, 
$z \cdot \Lambda_{j}- x \cdot \Lambda_{j} \in 
 Q_{+}:=\sum_{j \in I} \BZ_{\ge 0} \alpha_{j}$ if and only if 
$zW_{\Lambda_{j}} \le xW_{\Lambda_{j}}$ 
in the Bruhat order on $W/W_{\Lambda_{j}}$.
Thus, the inequality \eqref{eq:key1} seems closely related to 
\cite[Proposition~2.7]{Kam0} in the case of type $A$.
In fact, in the Appendix, we prove a (slightly) strengthened form 
(Proposition~\ref{prop:apdx}), which can be regarded as a generalization 
of \cite[Proposition~2.7]{Kam0}.
\end{rem}

\begin{proof}[Proof of Proposition~\ref{prop:key1}]
For $z,\,x \in W$ such that $z \le x$, 
there exists a sequence $z=x_{0} < x_{1} < \cdots < x_{s}=x$ 
of elements in $W$ such that $\ell(x_{t})=\ell(x_{t-1})+1$ 
for $1 \le t \le s$ by the chain property 
(see, for example, \cite[Theorem~2.2.6]{BB}). 
Hence we may assume that $\ell(x)=\ell(z)+1$. 

Let $x=s_{i_{1}}s_{i_{2}} \cdots s_{i_{k}}$ 
be a reduced expression of $x$. Because $z < x$ and $\ell(x)=\ell(z)+1$, 
it follows from the strong exchange property (see, for example, 
\cite[Theorem~1.4.3]{BB}) that $z$ has a reduced expression of the form: 
$z=s_{i_{1}} \cdots s_{i_{l-1}}s_{i_{l+1}} \cdots s_{i_{k}}$ 
for some $1 \le l \le k$. 

\paragraph{Case 1.}
Suppose that $l=1$; in this case, we have 
$x=s_{i_{1}}z > z=s_{i_{2}} \cdots s_{i_{k}}$. 
Then, we see from Proposition~\ref{prop:ehrig} that
\begin{equation*}
x^{-1} \cdot \mu_{x}=\wt (f_{i_{k}}^{\max} \cdots f_{i_{2}}^{\max}f_{i_{1}}^{\max}P), \qquad
z^{-1} \cdot \mu_{z}=\wt (f_{i_{k}}^{\max} \cdots f_{i_{2}}^{\max}P). 
\end{equation*}
Now we set $P'=P(\mu'_{\bullet}):=f_{i_{1}}^{\max}P \in \mv(\lambda)$, 
where $\mu'_{\bullet}=(\mu'_{w})_{w \in W}$ denotes the GGMS datum of $P'$. 
Then, again from Proposition~\ref{prop:ehrig}, we see that
%
%
\begin{equation} \label{eq:key1-a}
x^{-1} \cdot \mu_{x}=
\wt (f_{i_{k}}^{\max} \cdots f_{i_{2}}^{\max}\underbrace{f_{i_{1}}^{\max}P}_{=P'})=
z^{-1} \cdot \mu'_{z}
\end{equation}
since $z=s_{i_{2}} \cdots s_{i_{k}}$.
Because $P'=f_{i_{1}}^{\max}P \supset P$ by Fact~\ref{fact:fjP}, 
it follows immediately from \eqref{eq:poly} that 
$z^{-1} \cdot \mu_{z} - z^{-1} \cdot \mu'_{z} \in 
\sum_{j \in I} \BR_{\ge 0} h_{j}$. Combining this and \eqref{eq:key1-a}, 
we obtain
\begin{equation*}
z^{-1} \cdot \mu_{z} - x^{-1} \cdot \mu_{x} \in \sum_{j \in I} \BR_{\ge 0} h_{j}. 
\end{equation*}
Since $z^{-1} \cdot \mu_{z} - x^{-1} \cdot \mu_{x} \in Q^{\vee}$ 
by Remark~\ref{rem:GGMS-Q}, 
we conclude that $z^{-1} \cdot \mu_{z} - x^{-1} \cdot \mu_{x} \in Q^{\vee}_{+}$, 
which implies that $z^{-1} \cdot \mu_{z} \ge x^{-1} \cdot \mu_{x}$. 

\paragraph{Case 2.}
Suppose that $l \ge 2$. 
We set $P''=P(\mu''_{\bullet}):=
f_{i_{l-1}}^{\max} \cdots f_{i_{2}}^{\max}f_{i_{1}}^{\max}P \in \mv(\lambda)$, 
where $\mu''_{\bullet}=(\mu''_{w})_{w \in W}$ denotes the GGMS datum of $P''$, 
and set $x'':=s_{i_{l}}s_{i_{l+1}} \cdots s_{i_{k}}$, 
$z'':=s_{i_{l+1}} \cdots s_{i_{k}}$.
Then we see from Proposition~\ref{prop:ehrig} that
\begin{align*}
x^{-1} \cdot \mu_{x} & =
  \wt (f_{i_{k}}^{\max} \cdots 
       f_{i_{l+1}}^{\max}f_{i_{l}}^{\max}
       \underbrace{f_{i_{l-1}}^{\max} \cdots f_{i_{2}}^{\max}f_{i_{1}}^{\max}P}_{=P''}) \\[1.5mm]
& = \wt (f_{i_{k}}^{\max} \cdots f_{i_{l+1}}^{\max}f_{i_{l}}^{\max}P'')
  = (x'')^{-1} \cdot \mu_{x''}'',
\end{align*}
\begin{align*}
z^{-1} \cdot \mu_{z} & =
  \wt (f_{i_{k}}^{\max} \cdots f_{i_{l+1}}^{\max}
       \underbrace{f_{i_{l-1}}^{\max} \cdots f_{i_{2}}^{\max}f_{i_{1}}^{\max}P}_{=P''}) \\[1.5mm]
& = \wt (f_{i_{k}}^{\max} \cdots f_{i_{l+1}}^{\max}P'') = (z'')^{-1} \cdot \mu_{z''}''.
\end{align*}
Consequently, 
by applying the result in Case 1 (with $x=x''$ and $z=z''$), 
we obtain $(z'')^{-1} \cdot \mu_{z''}'' \ge (x'')^{-1} \cdot \mu_{x''}''$, and hence
\begin{equation*}
z^{-1} \cdot \mu_{z}=(z'')^{-1} \cdot \mu_{z''}'' \ge 
(x'')^{-1} \cdot \mu_{x''}''=x^{-1} \cdot \mu_{x},
\end{equation*}
as desired. This proves the proposition.
\end{proof}

%
\begin{prop} \label{prop:key}
Keep the setting above. 
Let $x \in W$, 
and let $x=s_{i_{1}}s_{i_{2}} \cdots s_{i_{k}}$ be 
an arbitrary reduced expression of $x$. 
Suppose that $f_{i_{k}}^{c_{k}} \cdots 
f_{i_{2}}^{c_{2}}f_{i_{1}}^{c_{1}}P \ne \bzero$ for some 
$c_{1},\,c_{2},\,\dots,\,c_{k} \in \BZ_{\ge 0}$. Then, 
%
%
\begin{equation} \label{eq:key}
\wt(f_{i_{k}}^{c_{k}} \cdots f_{i_{2}}^{c_{2}}f_{i_{1}}^{c_{1}}P) \ge 
\wt(f_{i_{k}}^{\max} \cdots f_{i_{2}}^{\max}f_{i_{1}}^{\max}P). 
\end{equation}
\end{prop}

\begin{proof}
We show the assertion by induction on $\ell(x)$. 
If $\ell(x)=0$ or $1$, then the assertion is obvious. 
Hence we assume that $\ell(x) \ge 2$. 

\paragraph{Case 1.} 
Suppose that $c_{1}=0$; 
in this case, we have $f_{i_{k}}^{c_{\ell}} \cdots f_{i_{2}}^{c_{2}}P=
f_{i_{k}}^{c_{\ell}} \cdots f_{i_{2}}^{c_{2}}f_{i_{1}}^{c_{1}}P \ne \bzero$. 
By the induction hypothesis, we have 
%
%
\begin{equation} \label{eq:key1-1}
\wt(f_{i_{k}}^{c_{\ell}} \cdots f_{i_{2}}^{c_{2}}P) \ge 
\wt(f_{i_{k}}^{\max} \cdots f_{i_{2}}^{\max}P).
\end{equation}
It follows from Proposition~\ref{prop:key1} 
(with $z=s_{i_{1}}x=s_{i_{2}} \cdots s_{i_{k}}$) that 
$(s_{i_{1}}x)^{-1} \cdot \mu_{s_{i_{1}}x} \ge x^{-1} \cdot \mu_{x}$.
Also, by Proposition~\ref{prop:ehrig},
\begin{equation*}
(s_{i_{1}}x)^{-1} \cdot \mu_{s_{i_{1}}x} = 
\wt(f_{i_{k}}^{\max} \cdots f_{i_{2}}^{\max}P), \qquad
x^{-1} \cdot \mu_{x}=
\wt(f_{i_{k}}^{\max} \cdots f_{i_{2}}^{\max}f_{i_{1}}^{\max}P).
\end{equation*}
Therefore, we obtain
%
%
\begin{equation} \label{eq:key1-2}
\wt(f_{i_{k}}^{\max} \cdots f_{i_{2}}^{\max}P) \ge 
\wt(f_{i_{k}}^{\max} \cdots f_{i_{2}}^{\max}f_{i_{1}}^{\max}P).
\end{equation}
Combining \eqref{eq:key1-1} and \eqref{eq:key1-2}, 
we conclude that 
\begin{equation*}
\wt(f_{i_{k}}^{c_{k}} \cdots f_{i_{2}}^{c_{2}}f_{i_{1}}^{c_{1}}P) 
= 
\wt(f_{i_{k}}^{c_{k}} \cdots f_{i_{2}}^{c_{2}}P)
\ge \wt(f_{i_{k}}^{\max} \cdots f_{i_{2}}^{\max}f_{i_{1}}^{\max}P),
\end{equation*}
as desired. 

\paragraph{Case 2.} 
Suppose that $c_{1} \ne 0$. We set $P':=f_{i_{1}}^{c_{1}}P$; 
note that $f_{i_{1}}^{\max}P'=f_{i_{1}}^{\max}P$ by the definition of 
$f_{i_{1}}^{\max}$, and that 
$f_{i_{k}}^{c_{k}} \cdots f_{i_{2}}^{c_{2}}P'=
f_{i_{k}}^{c_{k}} \cdots f_{i_{2}}^{c_{2}}f_{i_{1}}^{c_{1}}P \ne \bzero$ 
by our assumption. Then we obtain 
\begin{align*}
\wt(f_{i_{k}}^{c_{k}} \cdots f_{i_{2}}^{c_{2}}f_{i_{1}}^{c_{1}}P)
& = \wt (f_{i_{k}}^{c_{k}} \cdots f_{i_{2}}^{c_{2}}P') \\
& \ge \wt(f_{i_{k}}^{\max} \cdots f_{i_{2}}^{\max}f_{i_{1}}^{\max}P') 
  \quad \text{by the result in Case~1} \\
& = \wt(f_{i_{k}}^{\max} \cdots f_{i_{2}}^{\max}f_{i_{1}}^{\max}P),
\end{align*}
as desired. This proves the proposition. 
\end{proof}

%
\subsection{Proof of the main result.}
\label{subsec:prf-main}

We are now in a position to give a proof of our main result. 

\begin{proof}[Proof of Theorem~\ref{thm:main}]
We write the GGMS data of $P_{1} \in \mv(\lambda_{1})$, 
$P_{2} \in \mv(\lambda_{2})$, and $P \in \mv(\lambda)$, 
respectively, as: 
\begin{equation*}
\mu^{(1)}_{\bullet}=(\mu^{(2)}_{w})_{w \in W}, \quad
\mu^{(2)}_{\bullet}=(\mu^{(2)}_{w})_{w \in W}, \quad \text{and} \quad
\mu_{\bullet}=(\mu_{w})_{w \in W}. 
\end{equation*}
We know from Lemma~\ref{lem:Minkowski} that 
the Minkowski sum $P_{1}+P_{2}$ is the pseudo-Weyl polytope 
$P(\mu_{\bullet}^{(1)}+\mu_{\bullet}^{(2)})$ 
with GGMS datum $\mu_{\bullet}^{(1)}+\mu_{\bullet}^{(2)}=
 (\mu_{w}^{(1)}+\mu_{w}^{(2)})_{w \in W}$. 
Therefore, by Lemma~\ref{lem:bz}, we have
\begin{align*}
& P \subset P_{1}+P_{2} \quad \Leftrightarrow \\
& 
(M_{w \cdot \Lambda_{j}}=) \ 
 \pair{\mu_{w}}{w \cdot \Lambda_{j}} \ge 
 \pair{\mu_{w}^{(1)}+\mu_{w}^{(2)}}{w \cdot \Lambda_{j}} \  (=M_{w \cdot \Lambda_{j}}')
\quad \text{for all $w \in W$ and $j \in I$}, 
\end{align*}
where $M_{\bullet}=(M_{\gamma})_{\gamma \in \Gamma}$ and 
$M_{\bullet}'=(M_{\gamma}')_{\gamma \in \Gamma}$ are the BZ data of 
$P=P(\mu_{\bullet})$ and $P_{1}+P_{2}=
P(\mu_{\bullet}^{(1)}+\mu_{\bullet}^{(2)})$, respectively. 
Hence, in order to prove the inclusion $P \subset P_{1}+P_{2}$, 
it suffices to show that 
%
%
\begin{equation} \label{eq:main}
w^{-1} \cdot \mu_{w} - w^{-1} \cdot (\mu_{w}^{(1)}+\mu_{w}^{(2)}) \in
\sum_{j \in I}\BR_{\ge 0} h_{j}
\quad \text{for all $w \in W$}.
\end{equation}

Take (and fix) $w \in W$ arbitrarily, 
and let $w=s_{i_{1}}s_{i_{2}} \cdots s_{i_{k}}$ 
be a reduced expression of $w$. 
Then we see from Proposition~\ref{prop:ehrig} that 
$w^{-1} \cdot \mu_{w}=
 \wt (f_{i_{k}}^{\max} \cdots f_{i_{2}}^{\max}f_{i_{1}}^{\max}P)$.
Because $\iota_{\lambda}(P)=P_{1} \otimes P_{2}$ by our assumption, 
we infer that
\begin{equation*}
\wt (f_{i_{k}}^{\max} \cdots f_{i_{2}}^{\max}f_{i_{1}}^{\max}P)=
\wt (f_{i_{k}}^{\max} \cdots f_{i_{2}}^{\max}f_{i_{1}}^{\max}(P_{1} \otimes P_{2})).
\end{equation*}
Here, by repeated application of the tensor product rule for 
the action of lowering Kashiwara operators $f_{j}$ for $j \in I$ 
(see \cite[Chapitre~9, Exercice~9.1]{Kas}), we deduce that
\begin{equation*}
f_{i_{k}}^{\max} \cdots f_{i_{2}}^{\max}f_{i_{1}}^{\max}(P_{1} \otimes P_{2})
= 
(f_{i_{k}}^{c_{k}} \cdots f_{i_{2}}^{c_{2}}f_{i_{1}}^{c_{1}}P_{1})
\otimes 
(f_{i_{k}}^{\max} \cdots f_{i_{2}}^{\max}f_{i_{1}}^{\max}P_{2})
\end{equation*}
for some $c_{1},\,c_{2},\,\dots,\,c_{k} \in \BZ_{\ge 0}$; 
note that $f_{i_{k}}^{c_{k}} \cdots f_{i_{2}}^{c_{2}}f_{i_{1}}^{c_{1}}P_{1} \ne \bzero$ 
since the left-hand side of the equation above is not equal to $\bzero$ by 
the definition of $f_{j}^{\max}$, $j \in I$. Hence we have 
\begin{align*}
w^{-1} \cdot \mu_{w} 
 & = \wt(f_{i_{k}}^{\max} \cdots f_{i_{2}}^{\max}f_{i_{1}}^{\max}(P_{1} \otimes P_{2})) \\
 & = \wt (f_{i_{k}}^{c_{k}} \cdots f_{i_{2}}^{c_{2}}f_{i_{1}}^{c_{1}}P_{1})+
     \wt (f_{i_{k}}^{\max} \cdots f_{i_{2}}^{\max}f_{i_{1}}^{\max}P_{2}) \\
 & = \wt (f_{i_{k}}^{c_{k}} \cdots f_{i_{2}}^{c_{2}}f_{i_{1}}^{c_{1}}P_{1})+
     w^{-1} \cdot \mu^{(2)}_{w} \quad \text{by Proposition~\ref{prop:ehrig}}.
\end{align*}
Also, we see again by Proposition~\ref{prop:ehrig} that 
\begin{equation*}
w^{-1} \cdot (\mu^{(1)}_{w}+\mu^{(2)}_{w}) =
w^{-1} \cdot \mu^{(1)}_{w}+ w^{-1} \cdot \mu^{(2)}_{w} =
  \wt (f_{i_{k}}^{\max} \cdots f_{i_{2}}^{\max}f_{i_{1}}^{\max}P_{1}) + 
  w^{-1} \cdot \mu^{(2)}_{w}. 
\end{equation*}
From the above, we conclude that 
\begin{equation*}
w^{-1} \cdot \mu_{w} - w^{-1} \cdot (\mu_{w}^{(1)}+\mu_{w}^{(2)}) 
= \wt (f_{i_{k}}^{c_{k}} \cdots f_{i_{2}}^{c_{2}}f_{i_{1}}^{c_{1}}P_{1})-
  \wt (f_{i_{k}}^{\max} \cdots f_{i_{2}}^{\max}f_{i_{1}}^{\max}P_{1});
\end{equation*}
the latter element is contained in 
$Q^{\vee}_{+}=\sum_{j \in I}\BZ_{\ge 0} h_{j}$ 
by Proposition~\ref{prop:key} (with $\lambda=\lambda_{1}$ and $P=P_{1}$), 
and hence in $\sum_{j \in I}\BR_{\ge 0}h_{j}$. Thus we have proved 
\eqref{eq:main}, thereby completing the proof of the theorem. 
\end{proof}

\begin{rem}
In the course of the proof above, we showed that for each $w \in W$, 
$w^{-1} \cdot \mu_{w}-w^{-1} \cdot \mu^{(2)}_{w}$ is the weight of 
the element $f_{i_{k}}^{c_{k}} \cdots f_{i_{2}}^{c_{2}}f_{i_{1}}^{c_{1}}P_{1} 
\in \mv(\lambda_{1})$. Because the Weyl group $W$ acts on the crystal 
$\mv(\lambda_{1}) \cong \CB(\lambda_{1})$ in a canonical way, we see that 
$\mu_{w}-\mu^{(2)}_{w}$ is also the weight of some element in $\mv(\lambda_{1})$. 
Therefore, in the notation of \S\ref{subsec:geom}, we obain 
$\mu_{w}-\mu^{(2)}_{w} \in \Omega(\lambda_{1})$; 
here, we recall the (well-known) fact that 
$\Omega(\lambda_{1})=\Conv(W \cdot \lambda_{1}) \cap (\lambda_{1}+Q^{\vee})$.
\end{rem}

%
\appendix
\section{Appendix: A geometric proof of the inequality \eqref{eq:key1}.}
\label{sec:key1}

\setcounter{subsection}{1}
\setcounter{equation}{0}

The aim of the appendix is to give a purely geometric proof of 
the inequality \eqref{eq:key1} in Proposition~\ref{prop:key1}; 
in fact, we prove a (slightly) strengthened form of this inequality, 
which can be regarded as a generalization of 
\cite[Proposition~2.7]{Kam0} in the case of type $A$ 
to an arbitrary semisimple Lie algebra (see Remark~\ref{rem:key1a}).
Below we use the 
setting of \S\ref{subsec:geom}. 
Let $\lambda \in X_{\ast}(T)_{+} \subset \Fh_{\BR}$ 
be a dominant coweight, and $P=P(\mu_{\bullet}) \in \mv(\lambda)$ 
an (arbitrary) MV polytope with GGMS datum 
$\mu_{\bullet}=(\mu_{w})_{w \in W}$. 
%
%
\begin{prop}[cf. Proposition~\ref{prop:key1}] \label{prop:apdx}
Fix $j \in I$ arbitrarily. Let $x,\,z \in W$ be such that 
$zW_{\Lambda_{j}} \le xW_{\Lambda_{j}}$ in the Bruhat order on 
the cosets $W/W_{\Lambda_{j}}$ modulo the stabilizer $W_{\Lambda_{j}}$ 
of $\Lambda_{j}$ in $W$. Then, we have 
$\pair{\mu_{z}}{z \cdot \Lambda_{j}} \ge \pair{\mu_{x}}{x \cdot \Lambda_{j}}$, 
that is, $M_{z \cdot \Lambda_{j}} \ge M_{x \cdot \Lambda_{j}}$. 
\end{prop}

\begin{proof}
Note that by the definition of the Bruhat order on $W/W_{\Lambda_{j}}$, 
we may assume that $x=s_{\alpha}z$ for a positive root $\alpha \in \Delta_{+}$
such that $\pair{\alpha^{\vee}}{z \cdot \Lambda_{j}} \ge 0$, where $\alpha^{\vee} \in \Fh$ 
denotes the coroot corresponding to $\alpha$. We divide the proof into two parts
according as $\pair{\mu_{z}}{\alpha} < 0$ or $\pair{\mu_{z}}{\alpha} \ge 0$. 

\paragraph{Case 1.}
First we assume that $\pair{\mu_{z}}{\alpha} < 0$. Let $L_{\alpha}$ be the 
connected subgroup of $G$ generated by $T$ and the one-parameter unipotent
subgroups corresponding to $\pm\alpha$, and $L$ the derived (connected) 
subgroup of $L_{\alpha}$, which is isomorphic to either $SL_{2}(\BC)$ or 
$PGL_{2}(\BC)$; note that $L[t^{s_{\alpha} \cdot \mu_{z}}]=
L \dot{s}_{\alpha}[t^{\mu_{z}}]=L[t^{\mu_{z}}]$ since $\dot{s}_{\alpha} \in L$. 
Denote by $\bb=\bb(\mu_{\bullet}) \in \CZ(\lambda)$ the MV cycle corresponding 
to the MV polytope $P=P(\mu_{\bullet}) \in \mv(\lambda)$ under the bijection 
$\Phi_{\lambda}:\mv(\lambda) \rightarrow \CZ(\lambda)$ in Theorem~\ref{thm:Kam1}.
Then it follows that $[t^{\mu_{z}}] \in \bb$. Since $\bb \in \CZ(\lambda)$ is 
an irreducible component of 
$\ol{ G(\CO)[t^{\lambda}] \cap U(\CK)[t^{\mu_{e}}] } \subset \Gr$, 
it is stable under the action of $U(\CO)=G(\CO) \cap U(\CK)$ and of 
$(L \cap U)(\CO) \subset U(\CO)$. In particular, we have 
\begin{equation*}
(L \cap U)(\CO)[t^{\mu_{z}}] \subset U(\CO)[t^{\mu_{z}}] \subset \bb,
\end{equation*}
and hence
\begin{equation*}
\ol{ (L \cap U)(\CO)[t^{\mu_{z}}] } \subset 
\ol{ U(\CO)[t^{\mu_{z}}] } \subset \ol{\bb}=\bb.
\end{equation*}

Now we observe that the stabilizer 
$\Stab_{L}[t^{\mu_{z}}]$ of $[t^{\mu_{z}}] \in \Gr$ in $L$ 
is equal to the intersection 
$L \cap \Ad(t^{\mu_{z}})G(\CO)=
 L \cap t^{\mu_{z}}G(\CO)(t^{\mu_{z}})^{-1}$
in $G(\CK)$. Here, from the assumption that 
$\pair{\mu_{z}}{\alpha} < 0$, it is easily shown 
(by using the connectedness of the Borel subgroup 
$L \cap {}^{s_{\alpha}}\!B$ of $L$) that 
$L \cap \Ad(t^{\mu_{z}})G(\CO) = L \cap {}^{s_{\alpha}}\!B \subset L$. 
Moreover, thanks to the Bruhat decomposition:
\begin{equation*}
L=\bigl((L \cap U) \dot{s}_{\alpha} (L \cap B)\bigr) \sqcup (L \cap B), 
\end{equation*}
we deduce that 
\begin{equation*}
(L \cap U)(L \cap {}^{s_{\alpha}}\!B)=
(L \cap U) \dot{s}_{\alpha} (L \cap B) \dot{s}_{\alpha} \subset 
L \dot{s}_{\alpha}=L
\end{equation*}
is an open dense subset, and hence 
\begin{equation*}
(L \cap U)[t^{\mu_{z}}]=
(L \cap U)(L \cap {}^{s_{\alpha}}\!B)[t^{\mu_{z}}] \subset L[t^{\mu_{z}}]
\end{equation*}
is also an open dense subset (recall that 
$L \cap {}^{s_{\alpha}}\!B=\Stab_{L}[t^{\mu_{z}}]$). 
Therefore, we see that
\begin{align*}
[t^{s_{\alpha} \cdot \mu_{z}}] \in 
 L[t^{s_{\alpha} \cdot \mu_{z}}] = L[t^{\mu_{z}}] 
 = \ol{ (L \cap U)[t^{\mu_{z}}] }
 \subset \ol{ (L \cap U)(\CO)[t^{\mu_{z}}] } 
 \subset \ol{ U(\CO)[t^{\mu_{z}}] } \subset \bb.
\end{align*}
Also, by part (1) of Theorem~\ref{thm:Kam1}, we infer that 
$\bb \subset \ol{ S_{\mu_{s_{\alpha}z}}^{s_{\alpha}z} }=
\ol{ {}^{s_{\alpha}z} U(\CK)[t^{\mu_{s_{\alpha} z}}] }$. 
Consequently, we obtain 
$[t^{s_{\alpha} \cdot \mu_{z}}] \in 
 \ol{ S_{\mu_{s_{\alpha}z}}^{s_{\alpha}z} }$. 
Hence it follows from \eqref{eq:Snuw} that 
\begin{equation*}
(s_{\alpha}z)^{-1} \cdot (s_{\alpha} \cdot \mu_{z}) \ge 
(s_{\alpha}z)^{-1} \cdot \mu_{s_{\alpha}z}, 
\quad \text{i.e.,} \quad
z^{-1} \cdot \mu_{z} \ge (s_{\alpha}z)^{-1} \cdot \mu_{s_{\alpha}z}=
x^{-1} \cdot \mu_{x}.
\end{equation*}
Namely, we have shown that 
$z^{-1} \cdot \mu_{z}-x^{-1} \cdot \mu_{x} \in Q^{\vee}_{+}$, 
which, in particular, implies that for the fixed $j \in I$, 
\begin{equation*}
M_{z \cdot \Lambda_{j}}-M_{x \cdot \Lambda_{j}}=
\pair{\mu_{z}}{z \cdot \Lambda_{j}}-
\pair{\mu_{x}}{x \cdot \Lambda_{j}}=
\pair{z^{-1} \cdot \mu_{z}-x^{-1} \cdot \mu_{x}}{\Lambda_{j}} \ge 0,
\end{equation*}
as desired.

\paragraph{Case 2.} 
Next we assume that $\pair{\mu_{z}}{\alpha} \ge 0$. 
Recall that we have $(s_{\alpha}z)^{-1} \cdot \mu_{z} \ge 
(s_{\alpha}z)^{-1} \cdot \mu_{s_{\alpha}z}$ by the definition of 
GGMS data (see \eqref{eq:GGMS}). Also, we have
$s_{\alpha} \cdot \mu_{z}=\mu_{z}-\pair{\mu_{z}}{\alpha}\alpha^{\vee}$, 
where $\pair{\mu_{z}}{\alpha} \ge 0$ by our assumption.
Therefore, for the fixed $j \in I$ such that 
$\pair{\alpha^{\vee}}{z \cdot \Lambda_{j}} \ge 0$, we compute:
\begin{align*}
0 & \le 
 \Bpair{(s_{\alpha}z)^{-1} \cdot \mu_{z}-
   (s_{\alpha}z)^{-1} \cdot \mu_{s_{\alpha}z}}{\Lambda_{j}}=
 \pair{s_{\alpha} \cdot \mu_{z}}{z \cdot \Lambda_{j}}-
 \pair{\mu_{s_{\alpha}z}}{(s_{\alpha}z) \cdot \Lambda_{j}}\\
& =
 \Bpair{\mu_{z}-\pair{\mu_{z}}{\alpha}\alpha^{\vee}}{z \cdot \Lambda_{j}}-
 M_{(s_{\alpha}z) \cdot \Lambda_{j}}
 = M_{z \cdot \Lambda_{j}}-\pair{\mu_{z}}{\alpha}\pair{\alpha^{\vee}}{z \cdot \Lambda_{j}}-
   M_{(s_{\alpha}z) \cdot \Lambda_{j}},
\end{align*}
and hence obtain
\begin{equation*}
M_{x \cdot \Lambda_{j}}=
M_{(s_{\alpha}z) \cdot \Lambda_{j}} \le 
M_{z \cdot \Lambda_{j}}-
 \underbrace{\pair{\mu_{z}}{\alpha}}_{\ge 0}
 \underbrace{\pair{\alpha^{\vee}}{z \cdot \Lambda_{j}}}_{\ge 0} 
\le M_{z \cdot \Lambda_{j}},
\end{equation*}
as desired. This proves the proposition. 
\end{proof}


{\small
\setlength{\baselineskip}{13pt}
\renewcommand{\refname}{References}

}

\end{document}

%% file: figure-tensor003.tex
\unitlength 0.1in
\begin{picture}( 15.0000,  3.3500)( 15.5000, -7.0000)
%
\special{pn 8}%
\special{sh 0.600}%
\special{ar 1600 600 50 50  0.0000000 6.2831853}%
\put(30.0000,-7.0000){\makebox(0,0)[lt]{$\mu_{\wi{l}}=\mu_{\wi{l-1}\si{l}}$}}%
\put(23.0000,-4.5000){\makebox(0,0){$\Ni{l}$}}%
%
\special{pn 8}%
\special{sh 0.600}%
\special{ar 3000 600 50 50  0.0000000 6.2831853}%
%
\special{pn 8}%
\special{pa 1600 600}%
\special{pa 3000 600}%
\special{fp}%
\put(16.0000,-7.0000){\makebox(0,0)[lt]{$\mu_{\wi{l-1}}$}}%
\end{picture}%

%% file: figure-tensor001.tex
\unitlength 0.1in
\begin{picture}( 44.6000, 24.0000)(  4.4000,-28.0000)
%
\special{pn 8}%
\special{pa 4200 800}%
\special{pa 4200 400}%
\special{fp}%
%
\special{pn 8}%
\special{pa 4200 2400}%
\special{pa 4200 2800}%
\special{fp}%
%
\special{pn 8}%
\special{sh 0.600}%
\special{ar 4200 800 50 50  0.0000000 6.2831853}%
%
\special{pn 8}%
\special{sh 0.600}%
\special{ar 4200 2400 50 50  0.0000000 6.2831853}%
\put(35.0000,-16.0000){\makebox(0,0)[rb]{$\mu_{\wi{k+1}}=\mu_{\wi{k}s_{i}}$}}%
\put(43.0000,-24.5000){\makebox(0,0)[lt]{$\mu_{\wi{k}}=\mu_{\wj{k}}$}}%
%
\special{pn 8}%
\special{pa 4200 2400}%
\special{pa 4800 1600}%
\special{fp}%
\special{pa 4800 1600}%
\special{pa 4200 800}%
\special{fp}%
\special{pa 4200 800}%
\special{pa 3600 1600}%
\special{fp}%
\special{pa 3600 1600}%
\special{pa 4200 2400}%
\special{fp}%
%
\special{pn 8}%
\special{sh 0.600}%
\special{ar 4800 1600 50 50  0.0000000 6.2831853}%
%
\special{pn 8}%
\special{sh 0.600}%
\special{ar 3600 1600 50 50  0.0000000 6.2831853}%
\put(49.0000,-16.0000){\makebox(0,0)[lb]{$\mu_{\wj{k+1}}=\mu_{\wj{k}s_{j}}$}}%
\put(38.0000,-20.0000){\makebox(0,0)[rt]{$\Ni{k+1}$}}%
\put(38.0000,-12.0000){\makebox(0,0)[rb]{$\Ni{k+2}$}}%
\put(46.0000,-20.0000){\makebox(0,0)[lt]{$\Nj{k+1}$}}%
\put(46.0000,-12.0000){\makebox(0,0)[lb]{$\Nj{k+2}$}}%
\put(43.0000,-8.0000){\makebox(0,0)[lb]{\vertexa}}%
\end{picture}%

%% file: figure-tensor002.tex
\unitlength 0.1in
\begin{picture}( 45.5000, 24.2000)(  1.5000,-28.0000)
%
\special{pn 8}%
\special{pa 4000 806}%
\special{pa 3400 1206}%
\special{fp}%
%
\special{pn 8}%
\special{pa 3400 1206}%
\special{pa 3400 2006}%
\special{fp}%
%
\special{pn 8}%
\special{pa 3400 2006}%
\special{pa 4000 2406}%
\special{fp}%
%
\special{pn 8}%
\special{pa 4000 2406}%
\special{pa 4600 2006}%
\special{fp}%
%
\special{pn 8}%
\special{pa 4600 2006}%
\special{pa 4600 1206}%
\special{fp}%
%
\special{pn 8}%
\special{pa 4600 1206}%
\special{pa 4000 806}%
\special{fp}%
%
\special{pn 8}%
\special{pa 4000 800}%
\special{pa 4000 400}%
\special{fp}%
%
\special{pn 8}%
\special{pa 4000 2400}%
\special{pa 4000 2800}%
\special{fp}%
%
\special{pn 8}%
\special{sh 0.600}%
\special{ar 4000 800 50 50  0.0000000 6.2831853}%
%
\special{pn 8}%
\special{sh 0.600}%
\special{ar 4600 1200 50 50  0.0000000 6.2831853}%
%
\special{pn 8}%
\special{sh 0.600}%
\special{ar 4600 2000 50 50  0.0000000 6.2831853}%
%
\special{pn 8}%
\special{sh 0.600}%
\special{ar 4000 2400 50 50  0.0000000 6.2831853}%
%
\special{pn 8}%
\special{sh 0.600}%
\special{ar 3400 2000 50 50  0.0000000 6.2831853}%
%
\special{pn 8}%
\special{sh 0.600}%
\special{ar 3400 1200 50 50  0.0000000 6.2831853}%
\put(41.5000,-5.5000){\makebox(0,0)[lb]{\vertexb}}%
\put(37.0000,-9.5000){\makebox(0,0)[rb]{$\Ni{k+3}$}}%
\put(44.0000,-10.0000){\makebox(0,0)[lb]{$\Nj{k+3}$}}%
\put(48.5000,-16.0000){\makebox(0,0){$\Nj{k+2}$}}%
\put(31.5000,-16.0000){\makebox(0,0){$\Ni{k+2}$}}%
\put(37.0000,-22.5000){\makebox(0,0)[rt]{$\Ni{k+1}$}}%
\put(43.0000,-22.5000){\makebox(0,0)[lt]{$\Nj{k+1}$}}%
\put(47.0000,-13.5000){\makebox(0,0)[lb]{$\mu_{\wj{k+2}}=\mu_{\wj{k}s_js_i}$}}%
\put(47.0000,-21.5000){\makebox(0,0)[lb]{$\mu_{\wj{k+1}}=\mu_{\wj{k}s_j}$}}%
\put(43.0000,-26.5000){\makebox(0,0)[lt]{$\mu_{\wi{k}}=\mu_{\wj{k}}$}}%
\put(33.0000,-21.5000){\makebox(0,0)[rb]{$\mu_{\wi{k+1}}=\mu_{\wi{k}s_i}$}}%
\put(33.0000,-13.5000){\makebox(0,0)[rb]{$\mu_{\wi{k+2}}=\mu_{\wi{k}s_is_j}$}}%
%
\special{pn 8}%
\special{pa 4400 650}%
\special{pa 4100 750}%
\special{fp}%
\special{sh 1}%
\special{pa 4100 750}%
\special{pa 4170 748}%
\special{pa 4152 734}%
\special{pa 4158 710}%
\special{pa 4100 750}%
\special{fp}%
%
\special{pn 8}%
\special{pa 4400 2600}%
\special{pa 4100 2450}%
\special{fp}%
\special{sh 1}%
\special{pa 4100 2450}%
\special{pa 4152 2498}%
\special{pa 4148 2474}%
\special{pa 4170 2462}%
\special{pa 4100 2450}%
\special{fp}%
\end{picture}%

%% file: tensor-ver101005.bbl
\begin{thebibliography}{XXXX}

\bibitem[A]{A}
J. E. Anderson, 
A polytope calculus for semisimple groups, 
{\it Duke Math. J.} {\bf 116} (2003), 567--588.

\bibitem[BaG]{BG}
P. Baumann and S. Gaussent, 
On Mirkovi\'{c}-Vilonen cycles and crystal combinatorics, 
{\it Represent. Theory} {\bf 12} (2008), 83--130.

\bibitem[BeZ]{BeZe}
A. Berenstein and A. Zelevinsky, 
Total positivity in Schubert varieties, 
{\it Comment. Math. Helv.} {\bf 72} (1997), 128--166.

\bibitem[BjB]{BB}
A. Bj\"{o}rner and F. Brenti, 
``Combinatorics of Coxeter Groups'', 
Graduate Texts in Mathematics Vol.~231, 
Springer, New York, 2005.

\bibitem[BrFG]{BFG}
A. Braverman, M. Finkelberg, and D. Gaitsgory, 
Uhlenbeck spaces via affine Lie algebras, {\it in} 
``The Unity of Mathematics'' (P. Etingof et al., Eds.), 
pp.~17--135, Progr. Math. Vol.~244, 
Birkh\"auser, Boston, 2006.

\bibitem[BrG]{BrGa}
A. Braverman and D. Gaitsgory, 
Crystals via the affine Grassmannian, 
{\it Duke Math. J.} {\bf 107} (2001), 561--575.

\bibitem[E1]{E}
M. Ehrig, MV-polytopes via affine buildings, 
preprint 2009, arXiv:0903.3744v1.

\bibitem[E2]{E2}
M. Ehrig, Construction of MV-polytopes via LS-galleries, 
Dissertation, University of Cologne, 2008.

\bibitem[FZ]{FZ}
S. Fomin and A. Zelevinsky, Recognizing Schubert cells, 
{\it J. Algebraic Combin.} {\bf 12} (2000), 37--57. 

\bibitem[Kam1]{Kam2}
J. Kamnitzer, 
The crystal structure on the set of 
Mirkovi\'{c}-Vilonen polytopes, 
{\it Adv. Math.} {\bf 215} (2007), 66--93. 

\bibitem[Kam2]{Kam0}
J. Kamnitzer, Hives and the fibres of 
the convolution morphism, 
{\it Selecta Math. (N.S.)} {\bf 13} (2007), 483--496.

\bibitem[Kam3]{Kam1}
J. Kamnitzer, 
Mirkovi\'{c}-Vilonen cycles and polytopes, 
{\it Ann. of Math. (2)} {\bf 171} (2010), 731--777. 

\bibitem[Kas]{Kas}
M. Kashiwara, 
Bases cristallines des groupes quantiques, 
Cours Sp\'ecialis\'es Vol. 9, 
Soci\'et\'e Math\'ematique de France, Paris, 2002.

\bibitem[KNS]{KNS}
S. Kato, S. Naito, and D. Sagaki, 
Polytopal estimate of Mirkovi\'c-Vilonen polytopes 
lying in a Demazure crystal, preprint 2009, 
arXiv:0912.0586, to appear in {\it Adv. Math.}

\bibitem[MV1]{MV1}
I. Mirkovi\'c and K. Vilonen, 
Perverse sheaves on affine Grassmannians and Langlands duality,
{\it Math. Res. Lett.} {\bf 7} (2000), 13--24.

\bibitem[MV2]{MV2}
I. Mirkovi\'c and K. Vilonen, 
Geometric Langlands duality and representations of 
algebraic groups over commutative rings, 
{\it Ann. of Math.} (2) {\bf 166} (2007), 95--143. 

\bibitem[NS]{NS-dp}
S. Naito and D. Sagaki,
Mirkovi\'c-Vilonen polytopes lying in 
a Demazure crystal and an opposite Demazure crystal, 
{\it Adv. Math.}, {\bf 221} (2009), 1804--1842.

\end{thebibliography}
